\documentclass[leqno,11pt]{amsart}
\pdfoutput=1

\usepackage{amssymb}         
\usepackage{color}           
\usepackage[ascii]{inputenc} 
\usepackage{tikz}            
\usepackage{aliascnt}        
\usepackage{microtype}				 
\usepackage{graphics}        
\usepackage[bookmarksnumbered]{hyperref}
\renewcommand{\mathbb}{\mathbf}

\newtheorem*{thm-plain}{Theorem}

\newtheorem{thm}{Theorem}[section]

\newaliascnt{lem}{thm}
\newtheorem{lem}[lem]{Lemma}
\aliascntresetthe{lem}

\newaliascnt{cor}{thm}
\newtheorem{cor}[cor]{Corollary}
\aliascntresetthe{cor}

\newaliascnt{prp}{thm}

\aliascntresetthe{prp}

\newaliascnt{cnj}{thm}

\aliascntresetthe{cnj}

\newaliascnt{que}{thm}

\aliascntresetthe{que}

\newaliascnt{fct}{thm}

\aliascntresetthe{fct}

\theoremstyle{definition}

\newaliascnt{dfn}{thm}

\aliascntresetthe{dfn}

\newaliascnt{ntn}{thm}

\aliascntresetthe{ntn}

\newaliascnt{rem}{thm}
\newtheorem{rem}[rem]{Remark}
\aliascntresetthe{rem}

\newaliascnt{nte}{thm}

\aliascntresetthe{nte}

\newaliascnt{exl}{thm}

\aliascntresetthe{exl}

\DeclareMathOperator{\Cov}{Cov}
\newcommand{\Measc}{\mathrm{Meas}_\mathrm{c}} 
\DeclareMathOperator{\Var}{Var}

\newcommand{\E}[2]{\operatorname{E}_{#1}[#2]}
\newcommand{\nicedot}{\raisebox{0.35ex}{\tikz \fill (0,0) circle (1pt);}} 
\newcommand{\annrel}[2]{\stackrel{\makebox[0pt]{\scriptsize #1}}{#2}} 

\graphicspath{{images/}}

\numberwithin{equation}{section}
\allowdisplaybreaks[4] 

\hypersetup{pdfauthor={Thomas Bliem and Stavros Kousidis},pdftitle={Expected degree of weights in Demazure modules of sl2hat}}

\begin{document}

  \title{Expected Degree of Weights in Demazure Modules of \(\widehat{\mathfrak{sl}}_2\)}

  \author[Thomas Bliem]{Thomas Bliem$^*$}
  \address{Thomas Bliem, Department of Mathematics\\San Francisco State University\\1600 Holloway Ave\\San Francisco CA 94109\\United States}
  \email{\href{mailto:bliem@math.sfsu.edu}{bliem@math.sfsu.edu}}
  \thanks{$^*$Supported by the Deutsche Forschungsgemeinschaft, SPP 1388.}
    
  \author[Stavros Kousidis]{Stavros Kousidis$^\dagger$}
  \address{Stavros Kousidis, Mathematisches Institut, Universit\"at zu K\"oln, Weyertal 86-90, 50931 K\"oln, Germany}
\email{\href{mailto:skousidi@math.uni-koeln.de}{skousidi@math.uni-koeln.de}}
  \thanks{$^\dagger$Supported by the Deutscher Akademischer Austauschdienst and the Deutsche Forschungsgemeinschaft, SFB/TR 12.}

  \begin{abstract}
	  We compute the expected degree of a randomly chosen element in a basis of weight vectors in the Demazure module $V_w(\Lambda)$ of \(\widehat{\mathfrak{sl}}_2\).
	  We obtain \emph{en passant} a new proof of Sanderson's dimension formula for these Demazure modules.
  \end{abstract}
  
  \maketitle

  \tableofcontents

  \section{Introduction}
		
	The traditional way to study the dimensions of weight spaces of a given representation is by means of their generating function, the character.
	While the character comprehends all weight multiplicities, it may be difficult to extract meaningful information from it.
	For example, the characters of irreducible representations of semisimple Lie algebras are explicitly given by Weyl's character formula.
	But to estimate weight multiplicities in representations with a large highest weight, more meaningful information can possibly be obtained from the fact that for $N \to \infty$ the weight distribution of $V(N\lambda)$ converges weakly to an absolutely continuous measure with piecewise polynomial density.
	This fact is commonly not deduced from Weyl's character formula but by symplectic geometry \cite{MR674406}.
	Similarly, given the character of a representation, one can immediately write down the character of its tensor powers $T^N(V)$, namely $\mathrm{ch}_{T^N(V)} = (\mathrm{ch}_V)^N$.
	To extract meaningful information, one could interpret this probabilistically as a convolution product of measures, saying that weights in tensor powers are distributed like sums of independent random variables, identically distributed according to the weight distribution of $V$.
	Then, by the central limit theorem and careful analysis, one can derive statistical information and estimates of weight multiplicities in high tensor powers \cite{MR2102573}.

	In this article, we take a probabilistic point of view on weight multiplicities in Demazure modules of the affine Lie algebra $\mathfrak{g}$ of type $A_1^{(1)}$.
	As $\mathfrak{g}$ is of rank $2$, its weight distributions are discrete measures on the plane.
	Some examples are shown in \autoref{10-bilder}.
	If we think of $\mathfrak{g}$ as $\widehat{\mathfrak{sl}}_2$, the extended loop algebra of $\mathfrak{sl}_2$, natural coordinates on this plane are the the eigenvalue for $\left( \begin{smallmatrix} 1 & 0 \\ 0 & -1 \end{smallmatrix} \right) \in \mathfrak{sl}_2$, the \emph{finite weight}, and the eigenvalue for a scaling element, the \emph{degree}.
	This suggests to start the analysis with the description of the marginal distributions corresponding to the finite weight and the degree.

	The distribution of the finite weight is explicitly known.
	Namely, the finite weight is distributed like a sum of independent random variables, all but one distributed identically \cite{MR1407880}.
	This fact has received much attention subsequently and has been generalized considerably \cite{MR1620507,MR2235341}.
	Strangely, the other marginal distribution, the distribution of the degree, seems to have escaped attention so far.
	Some examples are shown in \autoref{degree-distribution-Lambda0} and \ref{degree-distribution-10Lambda0}.
	Note that, even though the pictures suggest that the central limit theorem holds, the degree is not distributed like the sum of independent random variables.
	The lack of attention is especially astonishing as in the case of irreducible highest weight representations the degree distribution yields Macdonald's identities for Dedekind's $\eta$-function \cite[\S 12.2]{MR1104219}.
	The purpose of this article is to determine the expected value of the degree distribution in the case of a Demazure module of $\widehat{\mathfrak{sl}}_2$ (\autoref{main-theorem-expected-degree}).
	Even though we prefer a probabilistic language, our result can equivalently be stated as follows:
	Let $\mathrm{ch}_V$ be the character of a Demazure module for $\widehat{\mathfrak{sl}}_2$.
	We compute the Taylor expansion at $0$ of the basic specialization of $\mathrm{ch}_V$ up to order $1$.

	We calculate the expected degree by induction on the number of Demazure operators in Demazure's character formula.
	The natural coordinates to do this are not the finite weight and the degree, but coordinates dual to the simple roots of $\mathfrak{g}$.
	The actual induction follows a snake-like pattern (\autoref{snake-recursion-0} and \ref{snake-recursion-1}).
	Our strategy dictates that we must express the expected degree of a weight in a given Demazure module in terms of statistical information about the weight distribution of the previous Demazure module.
	It turns out that this involves not only the expected value, but also a second moment (\autoref{recursion-expected-value}).
	For this reason, we cannot apply an induction argument at this point.
	By what appears to be a coincidence to us, the necessary second moment can be expressed purely in terms of the variance of the finite weight.
	This variance is known by \cite{MR1407880}, thereby yielding a recurrence relation and consequently an explicit formula.
	
	\begin{figure}
	\begin{tabular}{ccccc}
	\vtop{\null{\hbox{\includegraphics{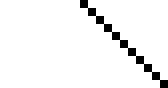}}}} &
	\vtop{\null{\hbox{\includegraphics{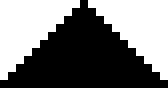}}}} &
	\vtop{\null{\hbox{\includegraphics{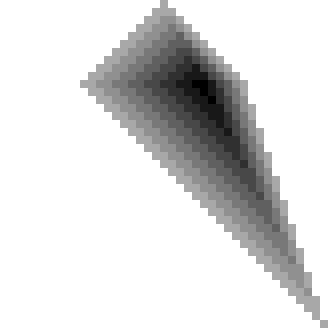}}}} &
	\vtop{\null{\hbox{\includegraphics{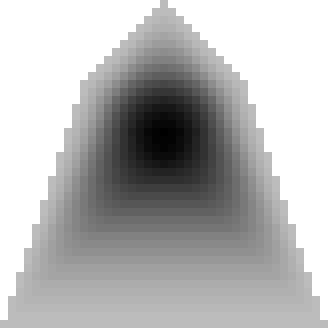}}}} &
	\vtop{\null{\hbox{\includegraphics{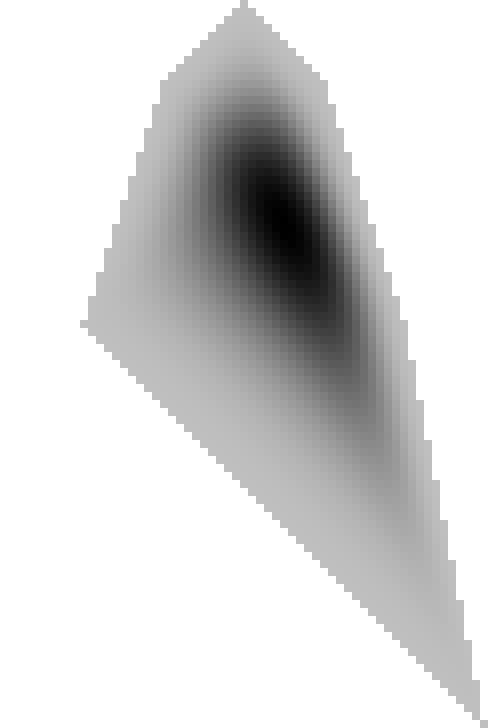}}}}
	\end{tabular}
	\begin{tabular}{ccc}
	\vtop{\null{\hbox{\includegraphics{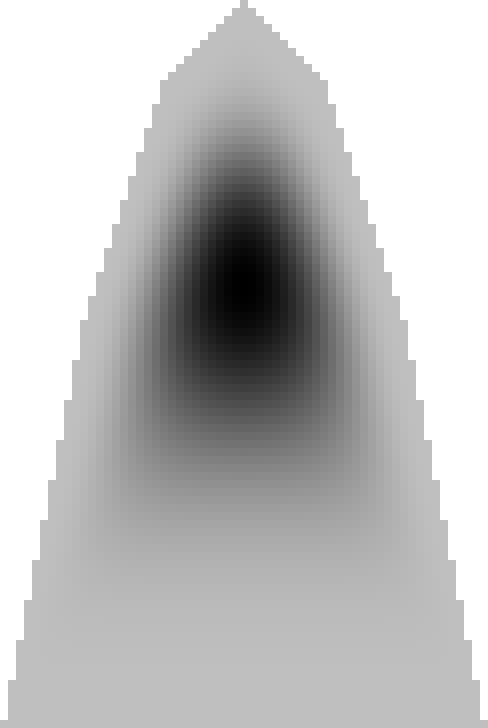}}}} &
	\vtop{\null{\hbox{\includegraphics{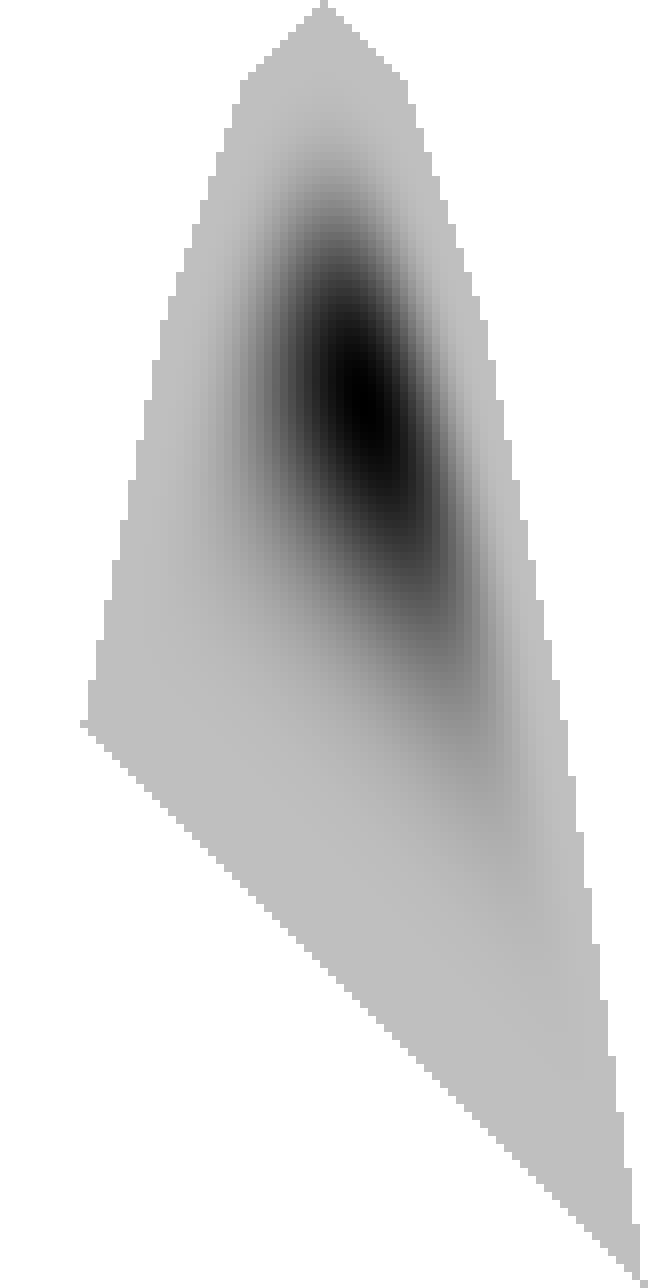}}}} &
	\vtop{\null{\hbox{\includegraphics{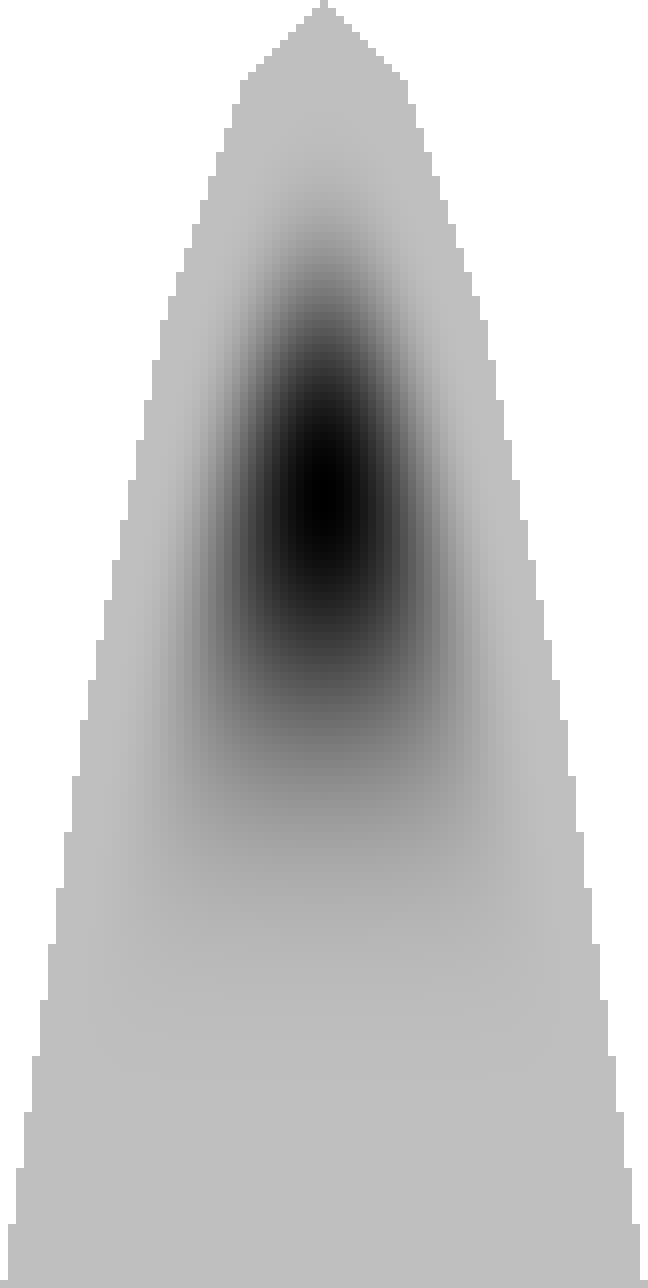}}}}
	\end{tabular}
	\caption{Weight distribution of $V_{w_{N,0}}(10\Lambda_0)$ for $N = 1, \ldots, 8$.
	The horizontal axis corresponds to the finite weight, the vertical axis to the degree.
	Light gray corresponds to the weight multiplicity $1$, black to the maximal occurring weight multiplicity in a given Demazure module.}
	\label{10-bilder}
	\end{figure}


	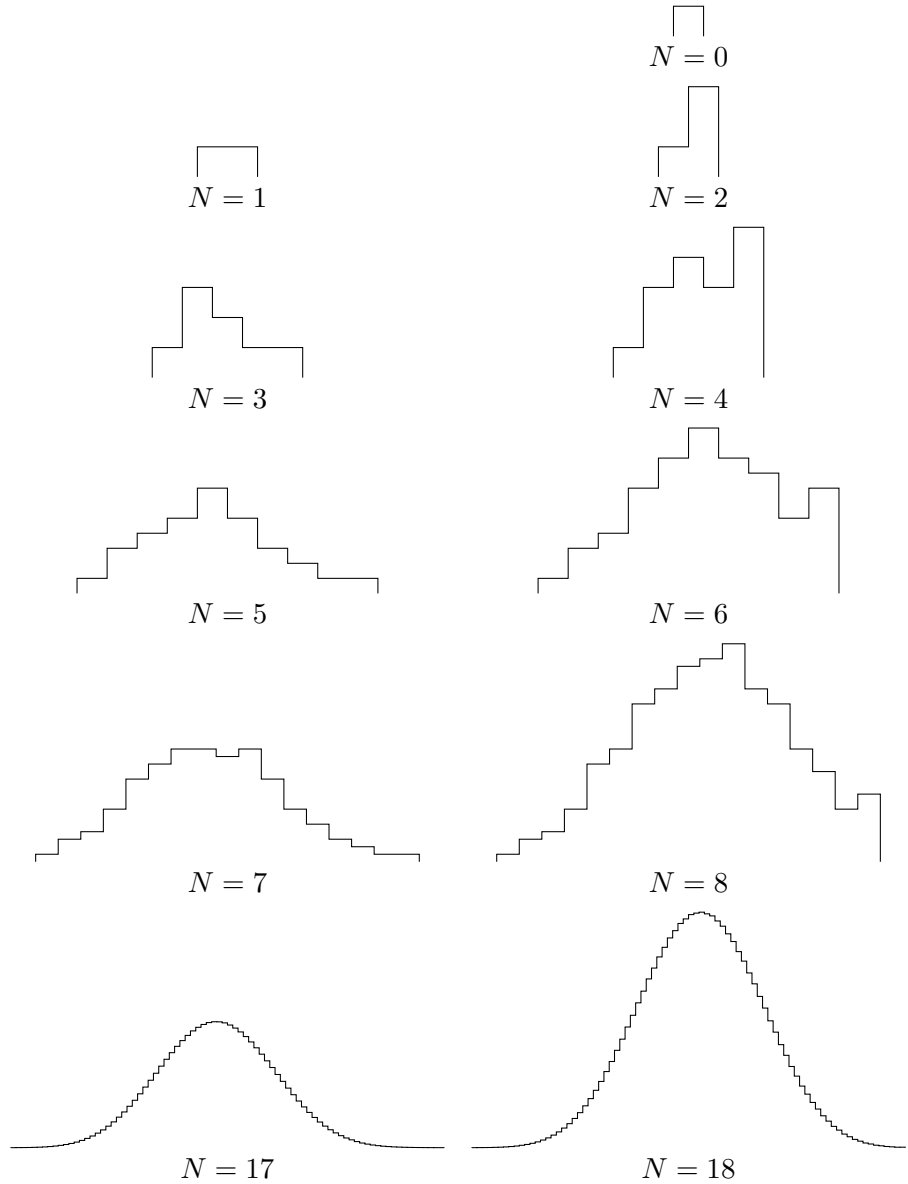
\begin{figure}
	\begin{tabular}{cc}
	&
	\begin{tikzpicture}[x=0.4cm,y=0.4cm]
		\draw (0,0) -- (0,1) -- (1,1) -- (1,0);
		\draw (.5,0) node[below]{$N=0$};
	\end{tikzpicture}
	\\
	\begin{tikzpicture}[x=0.4cm,y=0.4cm]
		\draw (0,0) -- (0,1) -- (1,1) -- (1,1) -- (2,1) -- (2,0);
		\draw (1,0) node[below]{$N=1$};
	\end{tikzpicture}
	&
	\begin{tikzpicture}[x=0.4cm,y=0.4cm]
		\draw (0,0) -- (0,1) -- (1,1) -- (1,3) -- (2,3) -- (2,0);
		\draw (1,0) node[below]{$N=2$};
	\end{tikzpicture}
	\\
	\begin{tikzpicture}[x=0.4cm,y=0.4cm]
		\draw (0,0) -- (0,1) -- (1,1) -- (1,3) -- (2,3) -- (2,2) -- (3,2) -- (3,1) -- (4,1) -- (4,1) -- (5,1) -- (5,0);
		\draw (2.5,0) node[below]{$N=3$};
	\end{tikzpicture}
	&
	\begin{tikzpicture}[x=0.4cm,y=0.4cm]
		\draw (0,0) -- (0,1) -- (1,1) -- (1,3) -- (2,3) -- (2,4) -- (3,4) -- (3,3) -- (4,3) -- (4,5) -- (5,5) -- (5,0);
		\draw (2.5,0) node[below]{$N=4$};
	\end{tikzpicture}
	\\
	\begin{tikzpicture}[x=0.4cm,y=0.2cm]
		\draw (0,0) -- (0,1) -- (1,1) -- (1,3) -- (2,3) -- (2,4) -- (3,4) -- (3,5) -- (4,5) -- (4,7) -- (5,7) -- (5,5) -- (6,5) -- (6,3) -- (7,3) -- (7,2) -- (8,2) -- (8,1) -- (9,1) -- (9,1) -- (10,1) -- (10,0);
		\draw (5,0) node[below]{$N=5$};
	\end{tikzpicture}
	&
	\begin{tikzpicture}[x=0.4cm,y=0.2cm]
		\draw (0,0) -- (0,1) -- (1,1) -- (1,3) -- (2,3) -- (2,4) -- (3,4) -- (3,7) -- (4,7) -- (4,9) -- (5,9) -- (5,11) -- (6,11) -- (6,9) -- (7,9) -- (7,8) -- (8,8) -- (8,5) -- (9,5) -- (9,7) -- (10,7) -- (10,0);
		\draw (5,0) node[below]{$N=6$};
	\end{tikzpicture}
	\\
	\begin{tikzpicture}[x=0.3cm,y=0.1cm]
		\draw (0,0) -- (0,1) -- (1,1) -- (1,3) -- (2,3) -- (2,4) -- (3,4) -- (3,7) -- (4,7) -- (4,11) -- (5,11) -- (5,13) -- (6,13) -- (6,15) -- (7,15) -- (7,15) -- (8,15) -- (8,14) -- (9,14) -- (9,15) -- (10,15) -- (10,11) -- (11,11) -- (11,7) -- (12,7) -- (12,5) -- (13,5) -- (13,3) -- (14,3) -- (14,2) -- (15,2) -- (15,1) -- (16,1) -- (16,1) -- (17,1) -- (17,0);
		\draw (8.5,0) node[below]{$N=7$};
	\end{tikzpicture}
	&
	\begin{tikzpicture}[x=0.3cm,y=0.1cm]
\draw (0,0) -- (0,1) -- (1,1) -- (1,3) -- (2,3) -- (2,4) -- (3,4) -- (3,7) -- (4,7) -- (4,13) -- (5,13) -- (5,15) -- (6,15) -- (6,21) -- (7,21) -- (7,23) -- (8,23) -- (8,26) -- (9,26) -- (9,27) -- (10,27) -- (10,29) -- (11,29) -- (11,23) -- (12,23) -- (12,21) -- (13,21) -- (13,15) -- (14,15) -- (14,12) -- (15,12) -- (15,7) -- (16,7) -- (16,9) -- (17,9) -- (17,0);
		\draw (8.5,0) node[below]{$N=8$};
	\end{tikzpicture}
	\\
	\begin{tikzpicture}[x=2pt,y=.01pt]
	\draw (0,0) -- (0,1) -- (1,1) -- (1,3) -- (2,3) -- (2,4) -- (3,4) -- (3,7) -- (4,7) -- (4,13) -- (5,13) -- (5,19) -- (6,19) -- (6,29) -- (7,29) -- (7,43) -- (8,43) -- (8,62) -- (9,62) -- (9,88) -- (10,88) -- (10,120) -- (11,120) -- (11,160) -- (12,160) -- (12,211) -- (13,211) -- (13,275) -- (14,275) -- (14,351) -- (15,351) -- (15,442) -- (16,442) -- (16,551) -- (17,551) -- (17,673) -- (18,673) -- (18,815) -- (19,815) -- (19,977) -- (20,977) -- (20,1156) -- (21,1156) -- (21,1353) -- (22,1353) -- (22,1568) -- (23,1568) -- (23,1795) -- (24,1795) -- (24,2040) -- (25,2040) -- (25,2295) -- (26,2295) -- (26,2554) -- (27,2554) -- (27,2819) -- (28,2819) -- (28,3088) -- (29,3088) -- (29,3348) -- (30,3348) -- (30,3598) -- (31,3598) -- (31,3838) -- (32,3838) -- (32,4056) -- (33,4056) -- (33,4252) -- (34,4252) -- (34,4423) -- (35,4423) -- (35,4559) -- (36,4559) -- (36,4667) -- (37,4667) -- (37,4738) -- (38,4738) -- (38,4770) -- (39,4770) -- (39,4766) -- (40,4766) -- (40,4727) -- (41,4727) -- (41,4645) -- (42,4645) -- (42,4530) -- (43,4530) -- (43,4384) -- (44,4384) -- (44,4208) -- (45,4208) -- (45,4005) -- (46,4005) -- (46,3783) -- (47,3783) -- (47,3541) -- (48,3541) -- (48,3290) -- (49,3290) -- (49,3036) -- (50,3036) -- (50,2772) -- (51,2772) -- (51,2512) -- (52,2512) -- (52,2256) -- (53,2256) -- (53,2006) -- (54,2006) -- (54,1767) -- (55,1767) -- (55,1544) -- (56,1544) -- (56,1335) -- (57,1335) -- (57,1144) -- (58,1144) -- (58,971) -- (59,971) -- (59,816) -- (60,816) -- (60,679) -- (61,679) -- (61,560) -- (62,560) -- (62,455) -- (63,455) -- (63,367) -- (64,367) -- (64,297) -- (65,297) -- (65,231) -- (66,231) -- (66,176) -- (67,176) -- (67,135) -- (68,135) -- (68,101) -- (69,101) -- (69,77) -- (70,77) -- (70,56) -- (71,56) -- (71,42) -- (72,42) -- (72,30) -- (73,30) -- (73,22) -- (74,22) -- (74,15) -- (75,15) -- (75,11) -- (76,11) -- (76,7) -- (77,7) -- (77,5) -- (78,5) -- (78,3) -- (79,3) -- (79,2) -- (80,2) -- (80,1) -- (81,1) -- (81,1) -- (82,1) -- (82,0);
		\draw (41,0) node[below]{$N=17$};
	\end{tikzpicture}
	&
	\begin{tikzpicture}[x=2pt,y=.01pt]
		\draw (0,0) -- (0,1) -- (1,1) -- (1,3) -- (2,3) -- (2,4) -- (3,4) -- (3,7) -- (4,7) -- (4,13) -- (5,13) -- (5,19) -- (6,19) -- (6,29) -- (7,29) -- (7,43) -- (8,43) -- (8,62) -- (9,62) -- (9,90) -- (10,90) -- (10,122) -- (11,122) -- (11,166) -- (12,166) -- (12,219) -- (13,219) -- (13,289) -- (14,289) -- (14,371) -- (15,371) -- (15,476) -- (16,476) -- (16,597) -- (17,597) -- (17,743) -- (18,743) -- (18,909) -- (19,909) -- (19,1107) -- (20,1107) -- (20,1326) -- (21,1326) -- (21,1585) -- (22,1585) -- (22,1858) -- (23,1858) -- (23,2173) -- (24,2173) -- (24,2508) -- (25,2508) -- (25,2883) -- (26,2883) -- (26,3264) -- (27,3264) -- (27,3687) -- (28,3687) -- (28,4110) -- (29,4110) -- (29,4570) -- (30,4570) -- (30,5014) -- (31,5014) -- (31,5486) -- (32,5486) -- (32,5928) -- (33,5928) -- (33,6392) -- (34,6392) -- (34,6803) -- (35,6803) -- (35,7225) -- (36,7225) -- (36,7589) -- (37,7589) -- (37,7944) -- (38,7944) -- (38,8220) -- (39,8220) -- (39,8492) -- (40,8492) -- (40,8665) -- (41,8665) -- (41,8827) -- (42,8827) -- (42,8884) -- (43,8884) -- (43,8920) -- (44,8920) -- (44,8856) -- (45,8856) -- (45,8779) -- (46,8779) -- (46,8587) -- (47,8587) -- (47,8389) -- (48,8389) -- (48,8100) -- (49,8100) -- (49,7808) -- (50,7808) -- (50,7426) -- (51,7426) -- (51,7066) -- (52,7066) -- (52,6618) -- (53,6618) -- (53,6206) -- (54,6206) -- (54,5729) -- (55,5729) -- (55,5288) -- (56,5288) -- (56,4805) -- (57,4805) -- (57,4376) -- (58,4376) -- (58,3903) -- (59,3903) -- (59,3498) -- (60,3498) -- (60,3075) -- (61,3075) -- (61,2710) -- (62,2710) -- (62,2333) -- (63,2333) -- (63,2027) -- (64,2027) -- (64,1713) -- (65,1713) -- (65,1463) -- (66,1463) -- (66,1208) -- (67,1208) -- (67,1005) -- (68,1005) -- (68,807) -- (69,807) -- (69,665) -- (70,665) -- (70,514) -- (71,514) -- (71,414) -- (72,414) -- (72,314) -- (73,314) -- (73,246) -- (74,246) -- (74,177) -- (75,177) -- (75,139) -- (76,139) -- (76,93) -- (77,93) -- (77,71) -- (78,71) -- (78,45) -- (79,45) -- (79,32) -- (80,32) -- (80,17) -- (81,17) -- (81,19) -- (82,19) -- (82,0);
		\draw (41,0) node[below]{$N=18$};
	\end{tikzpicture}
	\end{tabular}
	\caption{Degree distribution of $V_{w_{N,0}}(\Lambda_0)$ for $N = 0,1, \ldots, 8; 17, 18$.
	Degree $0$, the degree of the highest weight, is displayed on the left, the maximal occurring degree in a given Demazure module on the right.}
	\label{degree-distribution-Lambda0}
	\end{figure}
	
	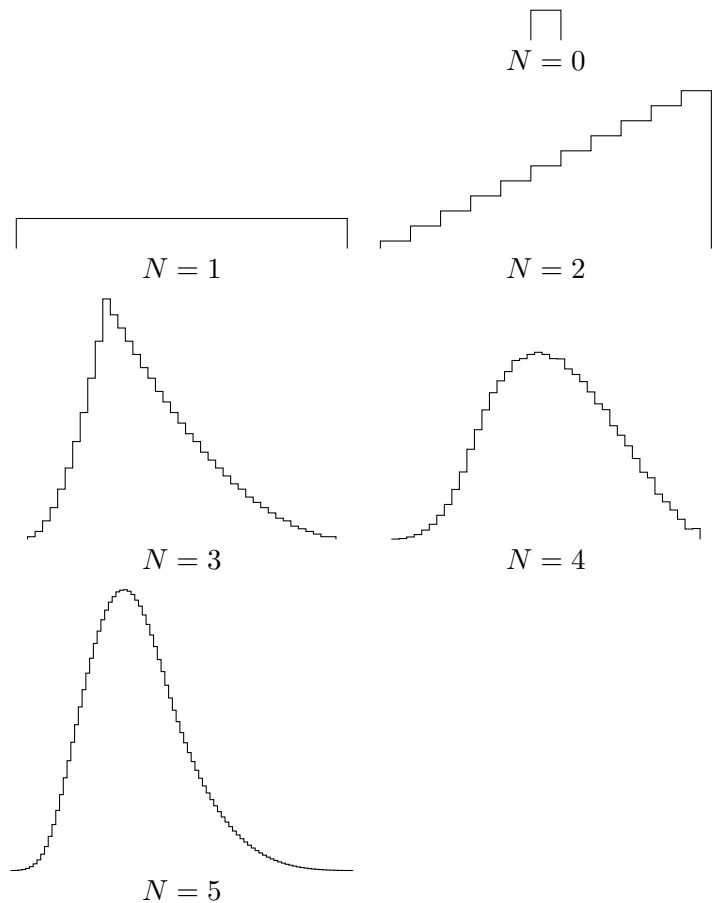
\begin{figure}
		\begin{tabular}{cc}
		&
		\begin{tikzpicture}[x=0.4cm,y=0.4cm]
			\draw (0,0) -- (0,1) -- (1,1) -- (1,0);
			\draw (0.5,0) node[below]{$N=0$};
		\end{tikzpicture}
		\\
		\begin{tikzpicture}[x=0.4cm,y=0.4cm]
			\draw (0,0) -- (0,1) -- (1,1) -- (1,1) -- (2,1) -- (2,1) -- (3,1) -- (3,1) -- (4,1) -- (4,1) -- (5,1) -- (5,1) -- (6,1) -- (6,1) -- (7,1) -- (7,1) -- (8,1) -- (8,1) -- (9,1) -- (9,1) -- (10,1) -- (10,1) -- (11,1) -- (11,0);
			\draw (5.5,0) node[below]{$N=1$};
		\end{tikzpicture}
		&
		\begin{tikzpicture}[x=0.4cm,y=.1cm]
			\draw (0,0) -- (0,1) -- (1,1) -- (1,3) -- (2,3) -- (2,5) -- (3,5) -- (3,7) -- (4,7) -- (4,9) -- (5,9) -- (5,11) -- (6,11) -- (6,13) -- (7,13) -- (7,15) -- (8,15) -- (8,17) -- (9,17) -- (9,19) -- (10,19) -- (10,21) -- (11,21) -- (11,0);
			\draw (5.5,0) node[below]{$N=2$};
		\end{tikzpicture}
		\\
		\begin{tikzpicture}[x=0.1cm,y=1pt]
			\draw (0,0) -- (0,1) -- (1,1) -- (1,3) -- (2,3) -- (2,7) -- (3,7) -- (3,12) -- (4,12) -- (4,19) -- (5,19) -- (5,27) -- (6,27) -- (6,37) -- (7,37) -- (7,48) -- (8,48) -- (8,61) -- (9,61) -- (9,75) -- (10,75) -- (10,91) -- (11,91) -- (11,85) -- (12,85) -- (12,80) -- (13,80) -- (13,75) -- (14,75) -- (14,70) -- (15,70) -- (15,65) -- (16,65) -- (16,61) -- (17,61) -- (17,56) -- (18,56) -- (18,52) -- (19,52) -- (19,48) -- (20,48) -- (20,44) -- (21,44) -- (21,40) -- (22,40) -- (22,37) -- (23,37) -- (23,33) -- (24,33) -- (24,30) -- (25,30) -- (25,27) -- (26,27) -- (26,24) -- (27,24) -- (27,21) -- (28,21) -- (28,19) -- (29,19) -- (29,16) -- (30,16) -- (30,14) -- (31,14) -- (31,12) -- (32,12) -- (32,10) -- (33,10) -- (33,8) -- (34,8) -- (34,7) -- (35,7) -- (35,5) -- (36,5) -- (36,4) -- (37,4) -- (37,3) -- (38,3) -- (38,2) -- (39,2) -- (39,1) -- (40,1) -- (40,1) -- (41,1) -- (41,0);
			\draw (20.5,0) node[below]{$N=3$};
		\end{tikzpicture}
		&
		\begin{tikzpicture}[x=0.1cm,y=.1pt]
			\draw (0,0) -- (0,1) -- (1,1) -- (1,3) -- (2,3) -- (2,9) -- (3,9) -- (3,18) -- (4,18) -- (4,35) -- (5,35) -- (5,57) -- (6,57) -- (6,91) -- (7,91) -- (7,132) -- (8,132) -- (8,189) -- (9,189) -- (9,255) -- (10,255) -- (10,341) -- (11,341) -- (11,415) -- (12,415) -- (12,490) -- (13,490) -- (13,555) -- (14,555) -- (14,600) -- (15,600) -- (15,635) -- (16,635) -- (16,677) -- (17,677) -- (17,684) -- (18,684) -- (18,700) -- (19,700) -- (19,708) -- (20,708) -- (20,700) -- (21,700) -- (21,684) -- (22,684) -- (22,683) -- (23,683) -- (23,645) -- (24,645) -- (24,624) -- (25,624) -- (25,597) -- (26,597) -- (26,558) -- (27,558) -- (27,513) -- (28,513) -- (28,491) -- (29,491) -- (29,430) -- (30,430) -- (30,394) -- (31,394) -- (31,354) -- (32,354) -- (32,306) -- (33,306) -- (33,254) -- (34,254) -- (34,233) -- (35,233) -- (35,171) -- (36,171) -- (36,142) -- (37,142) -- (37,111) -- (38,111) -- (38,76) -- (39,76) -- (39,39) -- (40,39) -- (40,41) -- (41,41) -- (41,0);
			\draw (20.5,0) node[below]{$N=4$};
		\end{tikzpicture}
		\\
		\begin{tikzpicture}[x=0.05cm,y=.02pt]
			\draw (0,0) -- (0,1) -- (1,1) -- (1,3) -- (2,3) -- (2,9) -- (3,9) -- (3,20) -- (4,20) -- (4,41) -- (5,41) -- (5,74) -- (6,74) -- (6,127) -- (7,127) -- (7,202) -- (8,202) -- (8,310) -- (9,310) -- (9,454) -- (10,454) -- (10,647) -- (11,647) -- (11,870) -- (12,870) -- (12,1139) -- (13,1139) -- (13,1436) -- (14,1436) -- (14,1755) -- (15,1755) -- (15,2079) -- (16,2079) -- (16,2428) -- (17,2428) -- (17,2760) -- (18,2760) -- (18,3097) -- (19,3097) -- (19,3423) -- (20,3423) -- (20,3735) -- (21,3735) -- (21,4018) -- (22,4018) -- (22,4298) -- (23,4298) -- (23,4531) -- (24,4531) -- (24,4746) -- (25,4746) -- (25,4929) -- (26,4929) -- (26,5079) -- (27,5079) -- (27,5183) -- (28,5183) -- (28,5274) -- (29,5274) -- (29,5306) -- (30,5306) -- (30,5314) -- (31,5314) -- (31,5286) -- (32,5286) -- (32,5224) -- (33,5224) -- (33,5117) -- (34,5117) -- (34,5004) -- (35,5004) -- (35,4837) -- (36,4837) -- (36,4658) -- (37,4658) -- (37,4457) -- (38,4457) -- (38,4238) -- (39,4238) -- (39,3992) -- (40,3992) -- (40,3765) -- (41,3765) -- (41,3507) -- (42,3507) -- (42,3266) -- (43,3266) -- (43,3034) -- (44,3034) -- (44,2818) -- (45,2818) -- (45,2611) -- (46,2611) -- (46,2418) -- (47,2418) -- (47,2233) -- (48,2233) -- (48,2062) -- (49,2062) -- (49,1898) -- (50,1898) -- (50,1747) -- (51,1747) -- (51,1602) -- (52,1602) -- (52,1469) -- (53,1469) -- (53,1342) -- (54,1342) -- (54,1226) -- (55,1226) -- (55,1115) -- (56,1115) -- (56,1014) -- (57,1014) -- (57,918) -- (58,918) -- (58,831) -- (59,831) -- (59,748) -- (60,748) -- (60,674) -- (61,674) -- (61,603) -- (62,603) -- (62,540) -- (63,540) -- (63,480) -- (64,480) -- (64,427) -- (65,427) -- (65,377) -- (66,377) -- (66,333) -- (67,333) -- (67,291) -- (68,291) -- (68,255) -- (69,255) -- (69,221) -- (70,221) -- (70,192) -- (71,192) -- (71,164) -- (72,164) -- (72,141) -- (73,141) -- (73,119) -- (74,119) -- (74,101) -- (75,101) -- (75,84) -- (76,84) -- (76,70) -- (77,70) -- (77,57) -- (78,57) -- (78,47) -- (79,47) -- (79,37) -- (80,37) -- (80,30) -- (81,30) -- (81,23) -- (82,23) -- (82,18) -- (83,18) -- (83,13) -- (84,13) -- (84,10) -- (85,10) -- (85,7) -- (86,7) -- (86,5) -- (87,5) -- (87,3) -- (88,3) -- (88,2) -- (89,2) -- (89,1) -- (90,1) -- (90,1) -- (91,1) -- (91,0);
				\draw (45.5,0) node[below]{$N=5$};
			\end{tikzpicture}
			\end{tabular}
			\caption{Degree distribution of $V_{w_{N,0}}(10\Lambda_0)$ for $N = 0, \ldots, 5$.}
		\label{degree-distribution-10Lambda0}
	\end{figure}
		
	\section{Notation and conventions}
	
	Let $\widehat{\mathfrak{sl}}_2$ be the affine Kac--Moody algebra corresponding to the extended Dynkin diagram of $\mathfrak{sl}_2$. Note that this can be realized as the extended loop algebra of $\mathfrak{sl}_2$ and in that spirit let $\hat{\mathfrak{b}} \supset \hat{\mathfrak{h}}$ be the Borel and Cartan subalgebra in $\widehat{\mathfrak{sl}}_2$ corresponding to their a priori fixed counterparts in $\mathfrak{sl}_2$. Denote by $\alpha_0$ and $\alpha_1$ the simple roots and by $\alpha_0^{\vee}, \alpha_1^{\vee}$ the simple coroots.
	 Let $s_0 , s_1$ be the simple reflections and $W^\mathrm{aff} = \langle s_0, s_1 \rangle$ the Weyl group. For $w \in W^\mathrm{aff}$ and a dominant integral weight \(\Lambda = m\Lambda_0 + n\Lambda_1\) denote the associated Demazure module by $V_w(\Lambda)$. All weights occuring in $V_w(\Lambda)$ are elements of the lattice \(\Gamma = \Lambda + \mathbb{Z}\alpha_0 + \mathbb{Z}\alpha_1 \subset \hat{\mathfrak{h}}^*\).
	We usually make dependencies on the Weyl group element $w$ explicit in the notation, while dependencies on the highest weight $\Lambda$ remain implicit.		
	Demazure's character formula allows the computation of the character of $V_w(\Lambda)$ by an iterated application of certain operators on the monomial $e^\Lambda$, as follows.
 	We introduce the convention that $\sum_{i=0}^{-1} a_i = 0$ and $\sum_{i=0}^k a_i = -a_{-1} - \cdots - a_{k+1}$ for $k < -1$. Note that this is natural in the sense that Gauss's summation formula $\sum_{i=0}^k i = \frac{k(k+1)}2$ extends to all $k \in \mathbb{Z}$, as does the identity $\sum_{i=0}^k 1 = k + 1$.
 	With this convention the Demazure operator $D_j$ associated with a simple reflection $s_j$ acts on monomials $e^\lambda$ as
	\[
		D_j e^\lambda 
		= \sum_{i=0}^{\langle \lambda, \alpha_j^\vee \rangle} e^{\lambda - i \alpha_j} .
 	\]
	For an arbitrary Weyl group element $w \in W^\mathrm{aff}$ we choose a reduced decomposition $w = s_{j_1} s_{j_2} \cdots s_{j_l}$ and set $D_w = D_{j_1}  D_{j_2} \cdots  D_{j_l}$. Demazure's character formula now states that the character of $V_w(\Lambda)$ can be computed as $\operatorname{ch}_{V_w(\Lambda)} = D_w e^\Lambda$.	
	
	Let $\Measc(\Gamma)$ denote the set of measures with finite support on the lattice $\Gamma$. Affiliated with the Demazure module $V_w(\Lambda)$ is its \emph{weight distribution} $\mu_w \in \Measc(\Gamma)$ given by $\mu_w = \sum_{\lambda \in \Gamma} \dim(V_w(\Lambda)_\lambda) \delta_\lambda$.
	Each $w \in W^{\mathrm{aff}}$ equals either $w_{N,0} = s_0^{(N \bmod 2)}(s_1 s_0)^{\lfloor N/2 \rfloor}$ or $w_{N,1} = s_1^{(N \bmod 2)}(s_0 s_1)^{\lfloor N/2 \rfloor}$ for some $N \in \mathbb{N}$.
	We abbreviate the corresponding weight distribution by $\mu_{N,j} = \mu_{w_{N,j}}$ for $j \in \{ 0, 1 \}$.
	The space $\Measc^{\mathbb{C}}(\Gamma)$ of complex measures with finite support on $\Gamma$ is isomorphic to the complex vector space generated by $\{ e^\lambda : \lambda \in \Gamma \}$ by the isomorphism $\delta_\lambda \mapsto e^\lambda$.
	The Demazure operators act on the latter hence via this isomorphism on the former by
	\[
		D_j \delta_\lambda 
		= \sum_{i=0}^{\langle \lambda, \alpha_j^\vee \rangle} \delta_{\lambda - i \alpha_j}
 	\]
	for $j \in \{0, 1\}$ and $\lambda \in \Gamma$.
	Demazure's character formula now becomes $\mu_w = D_w \delta_\Lambda$ for all $w \in W^\mathrm{aff}$.
	In particular, $\mu_{N,j} = D_{w_{N,j}} \delta_\Lambda$ for all $N \geq 0$ and $j \in \{0, 1\}$.
	
	Recall that the expected value of a function $f : \Gamma \to \mathbf{R}$ with respect to a nonzero measure $\mu \in \Measc(\Gamma)$ is $\E{\mu}{f} = \frac{1}{\mu(\Gamma)} \sum_{\lambda \in \Gamma} \mu(\{\lambda\}) f(\lambda)$.
	The covariance of two functions $f$ and $g$ is $\Cov_\mu(f, g) = \E{\mu}{(f - \E{\mu}{f})(g - \E{\mu}{g})}$, and the variance of $f$ is $\Var_\mu(f) = \Cov_\mu(f,f)$.	
	
	\section{Lemmata}

	For this section we fix a dominant integral weight \(\Lambda = m\Lambda_0 + n\Lambda_1\), and the lattice \(\Gamma = \Lambda + \mathbb{Z}\alpha_0 + \mathbb{Z}\alpha_1\).
	Define coordinates $a,b$ on $\Gamma$ such that \(\lambda=\Lambda -a(\lambda)\alpha_0 - b(\lambda)\alpha_1\) for all $\lambda \in \Gamma$.
	With these coordinates
	\begin{align}
		\label{0-height}
		\langle \; \nicedot \; , \alpha_0^\vee \rangle & = m - 2(a - b) \text{, and} \\
		\label{1-height}
		\langle \; \nicedot \;, \alpha_1^\vee \rangle & = n + 2(a - b).
	\end{align}
	We start by gathering information about the operation of the Demazure operators associated with simple reflections on measures in $\Measc(\Gamma)$.
	
  	\begin{lem}
		\label{expected-value-finite-weight}
		Let $\mu \in \Measc(\Gamma)$.
		Then \(\E{D_0 \mu}{a-b} = \frac{m}{2}\) and \(\E{D_1 \mu}{a-b} = -\frac{n}{2}\).%
		\footnote{Of course, we have to assume that $D_0\mu(\Gamma) \neq 0$ and $D_1\mu(\Gamma) \neq 0$, respectively, which we silently do here and in future similar situations.}
	\end{lem}
	
	\begin{proof}
		Let us first compute \(\E{D_0 \mu}{a-b} = \frac{m}{2}\). With $\mu = \sum_{\lambda \in \Gamma} p_\lambda \delta_\lambda$ we have
		\begin{align*}
			D_0 \mu (\Gamma) \cdot \E{D_0 \mu}{a-b}
				& = \sum_{\lambda \in \Gamma} p_\lambda \sum_{i=0}^{\langle \lambda,\alpha_0^{\vee} \rangle} (a-b)(\lambda - i \alpha_0) \\
				& = \sum_{\lambda \in \Gamma} p_\lambda \sum_{i=0}^{\langle \lambda,\alpha_0^{\vee} \rangle} ((a-b)(\lambda) + i) \\
				& = \sum_{\lambda \in \Gamma} p_\lambda \cdot (\langle \lambda,\alpha_0^{\vee} \rangle +1)((a-b)(\lambda) + \frac 12 \langle \lambda , \alpha_0^{\vee} \rangle) \\
				& \annrel{\eqref{0-height}}{=} \sum_{\lambda \in \Gamma} p_\lambda \cdot (\langle \lambda,\alpha_0^{\vee} \rangle +1) \cdot \frac m2 \\
				& = D_0 \mu (\Gamma) \cdot \frac m2.
		\end{align*}
		The last equation holds since each $\lambda \in \Gamma$ produces exactly $\langle \lambda,\alpha_0^{\vee} \rangle +1$ successors via the operation of $D_0$ on $\delta_\lambda$. To verify \(\E{D_1 \mu}{a-b} = -\frac{n}{2}\), one pursues the same computation.
	\end{proof}
	
	We want to prove \autoref{main-theorem-expected-degree} by induction on the length $N$ of the Weyl group element $w$. To that end we investigate how certain expected values change under the operation of the Demazure operators $D_0,D_1$.
	
	\begin{lem}
		\label{diagonal-trick}
		Let $\mu \in \Measc(\Gamma)$ and $k \geq 0$.
		Then
		\begin{align*}
			\E{D_1 \mu}{a^k}
				& = \frac{\mu(\Gamma)}{D_1 \mu (\Gamma)} \E{\mu}{a^k (n+1+2(a-b))} \text{, and}\\
			\E{D_0 \mu}{b^k}
				& = \frac{\mu(\Gamma)}{D_0 \mu (\Gamma)} \E{\mu}{b^k(m+1-2(a-b))}.
		\end{align*}
	\end{lem}
	
	\begin{proof}
		Let $\mu = \sum_{\lambda \in \Gamma} p_\lambda \delta_\lambda$. Then
		\begin{align*}
				D_1 \mu (\Gamma) \cdot \E{D_1 \mu}{a^k}
				& = \sum_{\lambda \in \Gamma} p_\lambda \sum_{i=0}^{\langle \lambda,\alpha_1^{\vee} \rangle} a^k (\lambda - i \alpha_1) \\
				& = \sum_{\lambda \in \Gamma} p_\lambda \sum_{i=0}^{\langle \lambda,\alpha_1^{\vee} \rangle} a^k (\lambda) \\
				& = \sum_{\lambda \in \Gamma} p_\lambda \cdot (\langle \lambda,\alpha_1^{\vee} \rangle +1) \cdot a^k (\lambda) \\
				& = \sum_{\lambda \in \Gamma} p_\lambda \cdot (n + 2(a-b)(\lambda) + 1) \cdot a^k (\lambda) \\
				& = \mu (\Gamma) \cdot \E{\mu}{a^k (n +1 + 2(a-b))}
		\end{align*}
		The computation is analogous for $\E{D_0 \mu}{b^k}$.
	\end{proof}
	
	By setting $k=0$ in \autoref{diagonal-trick} we obtain:
	
	\begin{cor}
		\label{recursion-total-mass}
		Let $\mu \in \Measc(\Gamma)$. Then, the total mass of $D_0 \mu$ and $D_1 \mu$ is given by
		\begin{align*}
			D_0\mu(\Gamma) & = \mu(\Gamma)(m+1 - 2 \E{\mu}{a-b}) \text{, and} \\
			D_1 \mu (\Gamma) & = \mu(\Gamma)(n+1+2\E{\mu}{a-b}).
		\end{align*}
	\end{cor}
	
	Resolving those equations for the weight distribution $\mu_w$ we derive the following dimension formulas for the Demazure module $V_w(\Lambda)$.
	
	\begin{cor}
	\label{total-mass}
		Let $N \geq 0$. Then the following dimension formulas hold:
		\begin{align*}
			\mu_{N,0}(\Gamma)
			&= \begin{cases} 1 & N=0 \\ (m+1)(m+n+1)^{N-1} & N \geq 1, \end{cases} \\
			\mu_{N,1}(\Gamma)
			&= \begin{cases} 1 & N=0 \\ (n+1)(m+n+1)^{N-1} & N \geq 1. \end{cases}
		\end{align*}
	\end{cor}
	
	\begin{proof}
		We restrict to the Weyl group elements starting in $s_0$, for the argumentation in the other case is similar. For $N=0$ one has $\mu_{0,0}(\Gamma) = \delta_\Lambda (\Gamma) = 1$, and if $N=1$, then $\mu_{1,0}(\Gamma) = D_0 \delta_\Lambda (\Gamma) = \langle \Lambda, \alpha_0^\vee \rangle + 1 = m+1$. We proceed by induction on $N \geq 2$.
		\begin{align*}
			\mu_{N,0}(\Gamma)
				& = \begin{cases} D_1 \mu_{N-1,0}(\Gamma) & \text{if $N$ even,} \\ D_0 \mu_{N-1,0}(\Gamma) & \text{if $N$ odd}\end{cases} \\
				& = \begin{cases} \mu_{N-1,0}(\Gamma)(n+1 + 2 \E{D_0 \mu_{N-2,0}}{a-b}) & \text{if $N$ even,} \\ \mu_{N-1,0}(\Gamma)(m+1 - 2 \E{D_1 \mu_{N-2,0}}{a-b}) & \text{if $N$ odd}\end{cases} \\
				& = \begin{cases} \mu_{N-1,0}(\Gamma)(n+1 + m) & \text{if $N$ even,} \\ \mu_{N-1,0}(\Gamma)(m+1 +n) & \text{if $N$ odd}\end{cases}
		\end{align*}
		The second equation follows from \autoref{recursion-total-mass} and the third by replacing the expected values by their actual values computed in \autoref{expected-value-finite-weight}. Hence by induction $\mu_{N,0}(\Gamma) = (m+1)(m+n+1)^{N-1}$.
	\end{proof}
	
	\begin{rem}
		\autoref{total-mass} is Sanderson's dimension formula \cite[Th.\ 1]{MR1389364}.
		Her original proof uses the path model for highest weight representations of Kac--Moody algebras.
	\end{rem}
	
	If we consider two consecutive applications of Demazure operators associated with simple reflections, \autoref{diagonal-trick} for $k = 1$ can be stated as follows.
		 
	\begin{lem}
		\label{recursion-expected-value}
		Let $\mu \in \Measc(\Gamma)$.
		Then we have the following equations:
		\begin{align*}
			\E{D_1 D_0 \mu}{a} & = \frac{D_0 \mu(\Gamma)}{D_1 D_0 \mu(\Gamma)}\Big((m+n+1)\E{D_0 \mu}{a} + 2 \Cov_{D_0 \mu}(a,a-b)\Big) \text{, and} \\
			\E{D_0 D_1 \mu}{b} & = \frac{D_1 \mu(\Gamma)}{D_0 D_1 \mu(\Gamma)}\Big((m+n+1)\E{D_1 \mu}{b} - 2 \Cov_{D_1 \mu}(b,a-b)\Big).
		\end{align*}
	\end{lem}
	
	\begin{proof}
		We prove the first assertion, the second being analogous. From \autoref{diagonal-trick} we obtain \(\E{D_1 D_0 \mu}{a} = \frac{D_0 \mu(\Gamma)}{D_1 D_0 \mu (\Gamma)} \E{D_0\mu}{a (n+1+2(a-b))}\).
			As $\E{D_0\mu}{a-b} = \frac m2$ by \autoref{expected-value-finite-weight}, the second factor of the right-hand side is
		\begin{align*}
			\E{D_0\mu}{a (n+1+2(a-b))}
				& = \E{D_0\mu}{a (n+1)} + 2 \E{D_0 \mu}{a(a-b)} \\
				& = (n+1)\E{D_0\mu}{a} + 2 \E{D_0 \mu}{a-b} \E{D_0 \mu}{a} \\ &\quad + \Cov_{D_0 \mu}(a,a-b) \\
				& = (m+n+1)\E{D_0\mu}{a} + \Cov_{D_0 \mu}(a,a-b). \qedhere
		\end{align*}
	\end{proof}

	Note that \autoref{recursion-expected-value} is not a recurrence relation for the expected degree, as covariances of the previous distribution are involved.
	
	\begin{lem}
		\label{covariance}
		Let $\mu \in \Measc(\Gamma)$. Then
		\begin{align*}
			\Cov_{D_0 \mu} (a,a-b) & = \Var_{D_0 \mu} (a-b)
			& \text{and } \Cov_{D_0 \mu} (b,a-b) & = 0 , \\
			\Cov_{D_1 \mu} (a,a-b) &= 0
			& \text{and } \Cov_{D_1 \mu} (b,a-b) & = -\Var_{D_1 \mu} (a-b) .
		\end{align*}
	\end{lem}
	
	\begin{proof}
		Let us first treat \(\Cov_{D_0 \mu}(a-b,b) = 0\).
		Indeed $\Cov_{D_0 \mu}(a-b,b) = \E{D_0 \mu}{(a-b)b} - \E{D_0 \mu}{a-b} \E{D_0 \mu}{b}$ and with $\mu = \sum_{\lambda \in \Gamma} p_\lambda \delta_\lambda$ we have
		\begin{align*}
				D_0 \mu (\Gamma) \cdot \E{D_0 \mu}{(a-b)b}
				& = \sum_{\lambda \in \Gamma} p_\lambda \sum_{i=0}^{\langle \lambda,\alpha_0^{\vee} \rangle} (a-b)(\lambda - i \alpha_0) \cdot b(\lambda - i \alpha_0) \\
				& = \sum_{\lambda \in \Gamma} p_\lambda \cdot b(\lambda) \sum_{i=0}^{\langle \lambda,\alpha_0^{\vee} \rangle} ((a-b)(\lambda) + i) \\
				&  \annrel{\eqref{0-height}}{=} \sum_{\lambda \in \Gamma} p_\lambda \cdot b(\lambda) \cdot (\langle \lambda,\alpha_0^{\vee} \rangle +1) \cdot \frac m2 \\
				& = \frac m2 \sum_{\lambda \in \Gamma} p_\lambda \sum_{i=0}^{\langle \lambda,\alpha_0^{\vee} \rangle} b(\lambda - i \alpha_0) \\
				& = \E{D_0 \mu}{a-b} \cdot D_0 \mu (\Gamma) \cdot \E{D_0 \mu}{b}.
		\end{align*}
		For the last equation see  \autoref{expected-value-finite-weight}.
		The same argument shows that $\Cov_{D_1\mu}(a,a-b) = 0$.
		The remaining claims follow from the bilinearity of the covariance.
		To be precise,
		\begin{align*}
			\Cov_{D_0 \mu} (a,a-b)
				& = \Cov_{D_0 \mu} ((a-b) + b,a-b) \\
				& =  \Cov_{D_0 \mu} (a-b,a-b) +  \Cov_{D_0 \mu} (b,a-b) \\
				& =  \Cov_{D_0 \mu} (a-b,a-b)
		\end{align*}
		The computation of $\Cov_{D_1 \mu} (b,a-b)$ is essentially the same.
	\end{proof}

	The variances appearing in \autoref{covariance} can be computed via Sanderson's formula \cite[Th.\ 1]{MR1407880} for the real character of the Demazure module $V_w(\Lambda)$.
	
	\begin{lem}
		\label{variance-finite-weights}
		For \(N \geq 1\) we have
		\begin{align*}
			\Var_{\mu_{N,0}} (a-b) &= \frac{m(m+2) + (N-1)(m+n)(m+n+2)}{12} \text{, and} \\
			\Var_{\mu_{N,1}} (a-b) &= \frac{n(n+2) + (N-1)(m+n)(m+n+2)}{12}.
		\end{align*}
	\end{lem}
	
	\begin{proof}
		We only compute $\mu_{N,0}$, the computation for $\mu_{N,1}$ being analogous.
		Let $q$ be a variable.
		For $k \geq 0$ we define the $q$-integer $[k]_q = \sum_{i=0}^{k-1} q^i$.
		Sanderson's formula \cite[Th.\ 1]{MR1407880} for the real character of $V_{w_{N,0}}(\Lambda)$ states that
		\[
			\sum_{\lambda \in \Gamma} \dim(V_{w_{N,0}}(\Lambda)_\lambda)q^{(a-b)(\lambda)}
			= q^{-(m+n)\lfloor N/2 \rfloor} [m+1]_q [m+n+1]_q^{N-1} .
		\]
		For $k \geq 0$ we define $\delta_{[k]}  = \sum_{i=0}^{k-1} \delta_i \in \Measc(\mathbb{Z})$.
		The linear map $\mathbb{C}[q, q^{-1}] \to \Measc^{\mathbb{C}}(\mathbb{Z})$ given by $q^k \mapsto \delta_k$ is an isomorphism of algebras, the multiplication of measures being convolution.
		This isomorphism maps $[k]_q$ to $\delta_{[k]}$ and hence
		\[
			\sum_{\lambda \in \Gamma} \dim(V_{w_{N,0}}(\Lambda)_\lambda)\delta_{(a-b)(\lambda)}
			= \delta_{-(m+n)\lfloor N/2 \rfloor} * \delta_{[m+1]} * \delta_{[m+n+1]}^{*(N-1)} .
		\]
		The measure on the left-hand side is by definition the push-forward measure $(a-b)_*\mu$.
		By straightforward computation
		$
			\Var(\delta_{[k]})
			= \frac 1k \sum_{i=0}^{k-1} \left( i - \frac{k-1}2 \right)^2
			=  \frac{(k-1)(k+1)}{12}
		$.
		Hence
		\begin{align*}
			\Var_{\mu_{N,0}}(a-b)
			  &= \Var((a-b)_*\mu) \\
			  &= \Var \left( \delta_{-(m+n)\lfloor N/2 \rfloor} * \delta_{[m+1]} * \delta_{[m+n+1]}^{*(N-1)} \right) \\
				& = \Var(\delta_{[m+1]}) + (N-1)\Var(\delta_{[m+n+1]}) \\
				& = \frac{m(m+2)}{12} + (N-1)\frac{(m+n)(m+n+2)}{12} . \qedhere
		\end{align*}
	\end{proof}
	
	Combining \autoref{recursion-expected-value}, \autoref{covariance}, and \autoref{variance-finite-weights}, we finally obtain recurrence relations.
	
	\begin{lem}[Recurrence relation]
		\label{recursion-formula}
		Let $N \geq 2$. Then the following recurrence relations hold:
		
		\begin{align}
			\label{recurrence-0-a}
			\E{\mu_{N,0}}{a}
			&= \frac{\mu_{N-1,0}(\Gamma)}{\mu_{N,0}(\Gamma)}
			\biggl(
				(m+n+1)\E{\mu_{N-1,0}}{a} \\ &\quad + \frac{m(m+2) + (N-2)(m+n)(m+n+2)}{6}
			\biggr)
			\text{ if $N$ even,} \notag
		\\
			\label{recurrence-0-b}
			\E{\mu_{N,0}}{b}
			&= \frac{\mu_{N-1,0}(\Gamma)}{\mu_{N,0}(\Gamma)}
			\biggl(
				(m+n+1)\E{\mu_{N-1,0}}{b} \\ &\quad + \frac{m(m+2) + (N-2)(m+n)(m+n+2)}{6}
			\biggr)
			\text{ if $N$ odd,} \notag
		\\
			\label{recurrence-1-b}
			\E{\mu_{N,1}}{b}
			&= \frac{\mu_{N-1,1}(\Gamma)}{\mu_{N,1}(\Gamma)}
			\biggl(
				(m+n+1)\E{\mu_{N-1,1}}{b} \\ &\quad + \frac{n(n+2) + (N-2)(m+n)(m+n+2)}{6}
			\biggr)
			\text{ if $N$ even,} \notag
		\\
			\label{recurrence-1-a}
			\E{\mu_{N,1}}{a}
			&= \frac{\mu_{N-1,1}(\Gamma)}{\mu_{N,1}(\Gamma)}
			\biggl(
				(m+n+1)\E{\mu_{N-1,1}}{a} \\ &\quad + \frac{n(n+2) + (N-2)(m+n)(m+n+2)}{6}
			\biggr)
		\text{ if $N$ odd.} \notag
		\end{align}	\end{lem}
	
	\begin{proof}
		The equations follow directly by replacing the covariance in \autoref{recursion-expected-value} with the variance as computed in \autoref{covariance}.
		Subsequently, substitute this variance by its value as computed in \autoref{variance-finite-weights}.
		Depending on the parity of $N$ one has to keep track during this procedure of the leftmost simple reflection in the Weyl group element $w_{N,j}$ defining the measure $\mu_{N,j} = \mu_{w_{N,j}}$.
	\end{proof}
	
  \section{Main theorems and conclusions}
  
  \begin{thm}[Expected degree]
  	\label{main-theorem-expected-degree}
	Let $\Lambda = m \Lambda_0 + n \Lambda_1$ be a dominant integral weight and $N \geq 1$. Choose a basis of weight vectors in the Demazure modules $V_{w_{N,0}}(\Lambda)$ and $V_{w_{N,1}}(\Lambda)$. Then the expected degrees of a randomly chosen basis element are given by the following formulas, respectively.
		\begin{align}
			\label{gDxcUasK}
			\E{\mu_{N,0}}{a}
			&= \frac{2(N-1)m(m+2)+(N-1)(N-2)(m+n)(m+n+2)}{12(m+n+1)} \\
			& \quad + \left\lfloor \frac{N-1}2 \right\rfloor \frac{m+n}2 + \frac m2 , \notag
			\\
			\label{TkzSAF6r}
			\E{\mu_{N,1}}{a}
			&= \frac{2(N-1)n(n+2)+(N-1)(N-2)(m+n)(m+n+2)}{12(m+n+1)} \\
			& \quad + \left\lfloor \frac N2 \right\rfloor \frac{m+n}2 . \notag
		\end{align}
  \end{thm}

	\begin{proof}
We start by showing \eqref{gDxcUasK}.
	One directly sees that $\E{\mu_{1,0}}{a} = \frac m2$.
	By \autoref{total-mass} we have $\frac{\mu_{N-1,0}(\Gamma)}{\mu_{N,0}(\Gamma)} = \frac 1{m+n+1}$ for $N \geq 2$.
	Hence by \eqref{recurrence-0-a} and \eqref{recurrence-0-b} we have
	\begin{equation}
		\label{yCbFJpUQ}
		\E{\mu_{N,0}}{a} = \E{\mu_{N-1,0}}{a} + 
		\frac{m(m+2) + (N-2)(m+n)(m+n+2)}{6(m+n+1)}
	\end{equation}
	for even $N \geq 2$, and
	\begin{equation}
		\label{mrDQMQUr}
		\E{\mu_{N,0}}{b} = \E{\mu_{N-1,0}}{b} + 
		\frac{m(m+2) + (N-2)(m+n)(m+n+2)}{6(m+n+1)}
	\end{equation}
	for odd $N \geq 2$.
	By \autoref{expected-value-finite-weight} we have
	\begin{equation}
		\label{awEccYNr}
		\E{\mu_{N,0}}{b} = \E{\mu_{N,0}}{a} + \frac n2
	\end{equation}
	for even $N \geq 2$, and
	\begin{equation}
		\label{YYihEz9e}
		\E{\mu_{N,0}}{a} = \E{\mu_{N,0}}{b} + \frac m2
	\end{equation}
	for odd $N \geq 2$.
	In order to recursively compute $\E{\mu_{N,0}}{a}$ from $\E{\mu_{1,0}}{a} = \frac m2$ we must apply -- in this order -- \eqref{yCbFJpUQ}, \eqref{awEccYNr}, \eqref{mrDQMQUr} and \eqref{YYihEz9e} periodically to compute $\E{\mu_{2,0}}{a}$, $\E{\mu_{2,0}}{b}$, $\E{\mu_{3,0}}{b}$, $\E{\mu_{3,0}}{a}$ etc.
		The reader is referred to \autoref{snake-recursion-0} for an illustration of these recursion steps.
	The contributions from \eqref{yCbFJpUQ} and \eqref{mrDQMQUr} add up to
	\[
		\sum_{i=2}^N \frac{m(m+2) + (i-2)(m+n)(m+n+2)}{6(m+n+1)}
	\]
	which is the first summand of \eqref{gDxcUasK} by Gauss's summation formula.
	The contributions from \eqref{awEccYNr} and \eqref{YYihEz9e} add up to $\left\lfloor \frac{N-1}2 \right\rfloor \frac {n+m}2$, which is the second summand of \eqref{gDxcUasK}.
	The third summand is the initial value of the recursion, $\E{\mu_{1,0}}{a} = \frac m2$.
	
	The proof of \eqref{TkzSAF6r} is similar.
	Here we deduce
		\begin{equation}
		\label{8i4JMzK7}
		\E{\mu_{N,1}}{b} = \E{\mu_{N-1,1}}{b} + 
		\frac{n(n+2) + (N-2)(m+n)(m+n+2)}{6(m+n+1)}
	\end{equation}
	for even $N \geq 2$, and
	\begin{equation}
		\label{4vZuHCon}
		\E{\mu_{N,1}}{a} = \E{\mu_{N-1,1}}{a} + 
		\frac{n(n+2) + (N-2)(m+n)(m+n+2)}{6(m+n+1)}
	\end{equation}
	for odd $N \geq 2$ from \eqref{recurrence-1-b} and \eqref{recurrence-1-a}, respectively.
	From \autoref{expected-value-finite-weight} we get
  	\begin{equation}
		\label{WtuRm8EK}
		\E{\mu_{N,1}}{a} = \E{\mu_{N,1}}{b} + \frac m2
	\end{equation}
	for even $N \geq 2$ and
	\begin{equation}
		\label{o6sAv6ke}
		\E{\mu_{N,1}}{b} = \E{\mu_{N,1}}{a} + \frac n2
	\end{equation}
	for odd $N \geq 2$.
	Now, starting from $\E{\mu_{1,1}}{b} = \frac n2$ we recursively compute $\E{\mu_{2,1}}{b}$, $\E{\mu_{2,1}}{a}$, $\E{\mu_{3,1}}{a}$, $\E{\mu_{3,1}}{b}$ etc.\ by periodic application -- again in this order -- of \eqref{8i4JMzK7}, \eqref{WtuRm8EK}, \eqref{4vZuHCon}, and \eqref{o6sAv6ke}, as illustrated in \autoref{snake-recursion-1}.
	The first summand of \eqref{TkzSAF6r} collects the contributions of the applications of \eqref{8i4JMzK7} and \eqref{4vZuHCon}.
	The second summand collects both the initial value and the contributions of the applications of \eqref{WtuRm8EK} and \eqref{o6sAv6ke}.
	\end{proof}
	
	\begin{figure}
		\[ \begin{array}{cc|ccc}
			N & w_{N,0} & \E{\mu_{N,0}}{a} & & \E{\mu_{N,0}}{b} \\ \hline
			1 & s_0 & \frac m2 & & 0 \\
			  & & \downarrow\!\makebox[0pt][l]{\scriptsize\eqref{yCbFJpUQ}} & & \\
			2 & s_1s_0 & * & \overset{\eqref{awEccYNr}}{\longrightarrow} & * \\
			  & & & & \downarrow\!\makebox[0pt][l]{\scriptsize\eqref{mrDQMQUr}} \\
			3 & s_0s_1s_0 & * & \overset{\eqref{YYihEz9e}}{\longleftarrow} & * \\
			  & & \downarrow\!\makebox[0pt][l]{\scriptsize\eqref{yCbFJpUQ}} & & \\
			4 & s_1s_0s_1s_0& * & \overset{\eqref{awEccYNr}}{\longrightarrow} & * \\
			  & & & & \downarrow\!\makebox[0pt][l]{\scriptsize\eqref{mrDQMQUr}}
		\end{array} \]
		\caption{Recursion steps for $\E{\mu_{N,0}}{a}$.}
		\label{snake-recursion-0}
	\end{figure}

	\begin{figure}
		\[ \begin{array}{cc|ccc}
			N & w_{N,1} & \E{\mu_{N,1}}{a} & & \E{\mu_{N,1}}{b} \\ \hline
			1 & s_1 & 0 & & \frac n2 \\
			  & & & & \downarrow\!\makebox[0pt][l]{\scriptsize\eqref{8i4JMzK7}} \\
			2 & s_0s_1 & * & \overset{\eqref{WtuRm8EK}}{\longleftarrow} & * \\
			  & & \downarrow\!\makebox[0pt][l]{\scriptsize\eqref{4vZuHCon}} & & \\
			3 & s_1s_0s_1 & * & \overset{\eqref{o6sAv6ke}}{\longrightarrow} & * \\
			  & & & & \downarrow\!\makebox[0pt][l]{\scriptsize\eqref{8i4JMzK7}} \\
			4 & s_0s_1s_0s_1& * & \overset{\eqref{WtuRm8EK}}{\longleftarrow} & * \\
			  & & \downarrow\!\makebox[0pt][l]{\scriptsize\eqref{4vZuHCon}} & &
		\end{array} \]
		\caption{Recursion steps for $\E{\mu_{N,1}}{a}$.}
		\label{snake-recursion-1}
	\end{figure}

	We compute the maximal occurring degree in a given Demazure module for comparison with the expected degree.
	
	\begin{lem}
		\label{maximal-degree}
		Let $\Lambda = m\Lambda_0 + n\Lambda_1$ be a dominant integral weight and $N \geq 0$. The highest degree of a weight is
		\begin{align*}
			A^{m,n}_{N,0}
			&= \left\lceil \frac N2 \right\rceil m + (m+n) \left( \left\lceil \frac N2 \right\rceil -1 \right) \left\lceil \frac N2 \right\rceil
			&&\text{ in $V_{w_{N,0}}(\Lambda)$, and} \\
			A^{m,n}_{N,1}			
			&= \left\lfloor \frac N2 \right\rfloor (m+2n) + (m+n) \left( \left\lfloor \frac N2 \right\rfloor -1 \right) \left\lfloor \frac N2 \right\rfloor
			&&\text{ in $V_{w_{N,1}}(\Lambda)$}.
		\end{align*}
	\end{lem}

	\begin{proof}
		By an easy induction on even $N$, one proves that the coefficients of $w_{N,0}\Lambda = \Lambda -x \alpha_0 - y \alpha_1$ satisfy $x = \sum_{i=0}^{N/2-1} (m + 2i(m+n))$ and $x-y = - \frac N2 (m+n)$.
		As the operator $D_1$ preserves the degree we have $A^{m,n}_{N-1,0} = A^{m,n}_{N,0}$.
		Hence replacing $\frac N2$ with $\left\lceil \frac N2 \right\rceil$ extends the formula to all $N$, thereby implying the lemma in the case of $V_{w_{N,0}}(\Lambda)$.	
		Note that the proof in the case of $V_{w_{N,1}}(\Lambda)$ is completely analogous if one starts with $\Lambda' = s_1 \Lambda = \Lambda -n\alpha_1$ instead. This first step introduces the $2n$ and switches the $\lceil \frac N2 \rceil$ to $\lfloor \frac N2 \rfloor$ in the formula claimed above.
	\end{proof}

	From \autoref{main-theorem-expected-degree} we can immediately derive asymptotic statements when the parameters $N,m$ or $n$ become large. Let us first treat the case when the length of the Weyl group element, that is $N$, tends to infinity.
	 
	\begin{cor}
		\label{asymptotics-N}
		Let $\Lambda = m \Lambda_0 + n \Lambda_1$ be a dominant integral weight and $j \in \{ 0,1 \}$. Then the limit ratio of the expected and maximal degree in $V_{w_{N,j}}(\Lambda)$, as $N$ tends to infinity, is given by
		\[
			\lim_{N \rightarrow \infty} \frac{\E{\mu_{N,j}}{a}}{A^{m,n}_{N,j}}
			=
			 \frac{m+n+2}{3(m+n+1)}.
		\]
	\end{cor}
	
	Similar asymptotic statements hold with respect to the coefficients of the fundamental weights $\Lambda_0, \Lambda_1$, i.e. $m$ and $n$, respectively.
	
	\begin{cor}
		\label{asymptotics-m-n}
		Let $N \geq 1$ and $j \in \{ 0,1 \}$. Then the limit ratios of the expected and maximal degree in $V_{w_{N,j}}(m\Lambda_0 + n\Lambda_1)$, as $m$ or $n$ tend to infinity, are given by
		\[
			\frac{\E{\mu_{N,0}}{a}}{A^{m,n}_{N,0}} \to
			\begin{cases}
				\frac{N^2 - N + 6 \left\lceil \frac N2 \right\rceil}{12 \left\lceil \frac N2 \right\rceil^2}
			 	& (m \to \infty,\ n \text{ fixed}) \\
				\frac{N^2 - 3N - 4 + 6 \left\lceil \frac N2 \right\rceil}{12 \left\lceil \frac N2 \right\rceil \left( \left\lceil \frac N2 \right\rceil -1 \right)}
				& (n \to \infty,\ m \text{ fixed}),
			\end{cases}
		\]
		and
		\[
			\frac{\E{\mu_{N,1}}{a}}{A^{m,n}_{N,1}} \to
			\begin{cases}
				\frac{N^2 - 3N +2 + 6 \left\lfloor \frac N2 \right\rfloor}{12 \left\lfloor \frac N2 \right\rfloor^2}
				& (m \to \infty,\ n \text{ fixed}) \\
				\frac{N^2 - N + 6 \left\lfloor \frac N2 \right\rfloor}{12 \left\lfloor \frac N2 \right\rfloor \left( \left\lfloor \frac N2 \right\rfloor +1 \right)}
				& (n \to \infty,\ m \text{ fixed}).
			\end{cases}
		\]
	\end{cor}

	
	It remains to extend \autoref{main-theorem-expected-degree} to the case of Demazure modules indexed by elements in the extended affine Weyl group $\widetilde{W}^\mathrm{aff} = \Sigma \ltimes W^\mathrm{aff}$. $\Sigma$ is the automorphism group of the Dynkin diagram of \(\widehat{\mathfrak{sl}}_2\). Note that $\Sigma = \langle \sigma \rangle$ and $\sigma^2 =1$ in the case of \(\widehat{\mathfrak{sl}}_2\). The element $\sigma$ maps $\Lambda=m\Lambda_0 + n\Lambda_1$ to $\sigma(\Lambda) = n\Lambda_0 + m \Lambda_1$ and the lattice $\Gamma$ to $\sigma(\Gamma) = \sigma(\Lambda) + \mathbb{Z} \alpha_0 + \mathbb{Z} \alpha_1$ via $\sigma(\Lambda + x \alpha_0 + y \alpha_1)=\sigma(\Lambda) + y \alpha_0 + x \alpha_1$. The Demazure operators $D_w$ with $w \in W^\mathrm{aff}$ are extended to Demazure operators indexed by elements in $\widetilde{W}^\mathrm{aff}$ by $D_\sigma e^\lambda = e^{\sigma(\lambda)}$ and $D_\sigma \delta_\lambda = \delta_{\sigma(\lambda)}$, respectively.

  	\begin{thm}
		\label{theorem-expected-degree-sigma}
		Let \(\Lambda = m \Lambda_0 + n \Lambda_1\) be a dominant integral weight and $N \geq 1$. Consider the extended affine Weyl group $\widetilde{W}^{\mathrm{aff}}$ and its elements $\sigma s_0$ and $\sigma s_1$, where $\sigma$ is the non-trivial automorphism of the Dynkin diagram of \(\widehat{\mathfrak{sl}}_2\). Choose a basis of weight vectors in the Demazure modules \(V_{(\sigma s_0)^N}(\Lambda)\) and $V_{(\sigma s_1)^N}(\Lambda)$. Then the expected degrees of a randomly chosen basis element are given by the following formulas, respectively.
		\begin{align*}
			\E{\mu_{(\sigma s_0)^N}}{a}
			&= \frac{2(N-1)m(m+2)+(N-1)(N-2)(m+n)(m+n+2)}{12(m+n+1)} \\
			& \quad + 	\left\lfloor \frac {N-1}2 \right\rfloor \frac n2
			+ \left\lfloor \frac N2 \right\rfloor \frac m2 , \\
			\E{\mu_{(\sigma s_1)^N}}{a}
			&= \frac{2(N-1)n(n+2)+(N-1)(N-2)(m+n)(m+n+2)}{12(m+n+1)} \\
			& \quad + 	\left\lceil \frac {N-1}2 \right\rceil \frac m2
			+ \left\lceil \frac N2 \right\rceil \frac n2 .
		\end{align*}
	\end{thm}
	
	\begin{proof}
		Let us first treat \(V_{(\sigma s_0)^N}(\Lambda)\). We have $(\sigma s_0)^N = \sigma^{(N \bmod 2)}w_{N,0}$. If $N$ is even, the expected degree is given by $\E{\mu_{N,0}}{a}$, and a slight modification of \eqref{gDxcUasK} in \autoref{main-theorem-expected-degree} immediately proves the claim. If $N$ is odd, note that $\sigma s_0 = s_1 \sigma$ implies $\sigma w_{N,0} = w_{N,1} \sigma$. Since $V_{w_{N,1}\sigma}(\Lambda) = V_{w_{N,1}}(\sigma(\Lambda))$ we now have to consider the weight distribution $\mu_{N,1}$ of  $V_{w_{N,1}}(\sigma(\Lambda))$ which is supported on the lattice $\Gamma' = \sigma(\Lambda) +\mathbb{Z} \alpha_0 + \mathbb{Z} \alpha_1$. Therefore the claimed formula follows from the second part, that is \eqref{TkzSAF6r}, of \autoref{main-theorem-expected-degree}, now applied with the highest weight $\sigma(\Lambda) = n\Lambda_0 + m\Lambda_1$. For the Demazure module $V_{(\sigma s_1)^N}(\Lambda)$ the situation is completely analogous. Here one notes that $(\sigma s_1)^N = \sigma^{(N \bmod 2)}w_{N,1}$ and $\sigma w_{N,1} = w_{N,0} \sigma$. The interesting case is again for $N$ odd. Now one has to consider the weight distribution $\mu_{N,0}$ of  $V_{w_{N,0}}(\sigma(\Lambda))$ to compute the expected degree of a randomly chosen basis weight vector in  $V_{(\sigma s_1)^N}(\Lambda)$.
	\end{proof}
	
	From \autoref{theorem-expected-degree-sigma} we can derive similar asymptotic statements as before. First let us compute the maximal occurring degree.
	
	\begin{lem}
		\label{maximal-degree-sigma}
		Let $\Lambda = m\Lambda_0 + n\Lambda_1$ be a dominant integral weight and $N \geq 0$. Consider the Demazure modules  \(V_{(\sigma s_0)^N}(\Lambda)\) and $V_{(\sigma s_1)^N}(\Lambda)$. The highest degree of a weight is
		\begin{align*}
			B^{m,n}_{N,0}
			&= \left\lfloor \frac N2 \right\rfloor m + (m+n) \left\lfloor \frac{N-1}2 \right\rfloor \left\lfloor \frac N2 \right\rfloor
			&&\text{ in $V_{(\sigma s_0)^N}(\Lambda)$, and} \\
			B^{m,n}_{N,1}			
			&= \left\lceil \frac N2 \right\rceil n + (m+n) \left\lfloor \frac N2 \right\rfloor \left\lceil \frac N2\right\rceil
			&&\text{ in $V_{(\sigma s_1)^N}(\Lambda)$}.
		\end{align*}
	\end{lem}
	
	\begin{proof}
		Both claims can be derived from \autoref{maximal-degree} and we will demonstrate this in the case of the Demazure module \(V_{(\sigma s_0)^N}(\Lambda)\). Note that $(\sigma s_0)^N = w_{N,0}$ if $N$ is even, and equals $w_{N,1}\sigma$ when $N$ is odd, and $\sigma(\Lambda) = n \Lambda_0 + m \Lambda_1$. By \autoref{maximal-degree} the maximal degree of a weight in  $V_{w_{N,0}}(\Lambda)$ is $\lceil N/2 \rceil m + (m+n) (\lceil N/2 \rceil -1) \lceil N/2 \rceil$, and is equal to $\lfloor N/2 \rfloor (n+2m) + (m+n) (\lfloor N/2 \rfloor -1) \lfloor N/2 \rfloor$ in  $V_{w_{N,1}}(\sigma(\Lambda))$. The formula claimed above for \(V_{(\sigma s_0)^N}(\Lambda)\) unifies those case considerations.
	\end{proof}
	
	As expected, for the Demazure modules $V_{(\sigma s_0)^N}(\Lambda)$ and $V_{(\sigma s_1)^N}(\Lambda)$ one obtains the same limit ratio for large $N$ as in \autoref{asymptotics-N}. But the limit ratios with respect to the coefficients of the fundamental weights $\Lambda_0$ and $\Lambda_1$ are slightly different.
	
	\begin{cor}
		\label{asymptotics-m-n-sigma}
		Let $N \geq 1$ and $j \in \{ 0,1 \}$. Then the limit ratios of the expected and maximal degree in $V_{(\sigma s_j)^N}(m\Lambda_0 + n\Lambda_1)$, as $m$ or $n$ tend to infinity, are given by
		\[
			\frac{\E{\mu_{(\sigma s_0)^N}}{a}}{B^{m,n}_{N,0}} \to
			\begin{cases}
				\frac{N^2 - N + 6 \left\lfloor \frac N2 \right\rfloor}{12 \left\lfloor \frac N2 \right\rfloor \left\lceil \frac N2 \right\rceil}
			 	& (m \to \infty,\ n \text{ fixed}) \\
				\frac{N^2 - 3N - 4 + 6 \left\lceil \frac N2 \right\rceil}{12 \left\lfloor \frac N2 \right\rfloor \left (\left\lceil \frac N2 \right\rceil -1 \right)}
				& (n \to \infty,\ m \text{ fixed}),
			\end{cases}
		\]
		and
		\[
			\frac{\E{\mu_{(\sigma s_1)^N}}{a}}{B^{m,n}_{N,1}} \to
			\begin{cases}
				\frac{N^2 - 3N +2 + 6 \left\lfloor \frac N2 \right\rfloor}{12 \left\lfloor \frac N2 \right\rfloor \left\lceil \frac N2 \right\rceil}
				& (m \to \infty,\ n \text{ fixed}) \\
				\frac{N^2 - N + 6 \left\lceil \frac N2 \right\rceil}{12 \left\lceil \frac N2 \right\rceil \left( \left\lfloor \frac N2 \right\rfloor +1 \right)}
				& (n \to \infty,\ m \text{ fixed}).
			\end{cases}
		\]
	\end{cor}
	
	\section{Further questions}
		
	The first natural question which comes to mind is, what is so special about $\widehat{\mathfrak{sl}}_2$? Based on \cite[Th.~1,~(1.2)]{MR1407880} one is immediately tempted to include type $A^{(2)}_2$ in the considerations presented here.
	But it is not clear to us how to adapt the covariance-variance reduction (\autoref{recursion-expected-value} and \autoref{covariance}) to this case.
				
		In higher rank examples we only know the distribution of the finite weight for specific elements of the Weyl group by \cite{MR2235341}. Hence we would have to restrict our attention to sequences consisting only of those elements to be able to develop a recurrence relation. Even in those cases one runs into similar problems as for type $A^{(2)}_2$ when trying to adapt the covariance-variance reduction.
		
		Another potential continuation of the present discussion would be to compute the variance of the degree distribution. For example, one can readily express the variance of the degree in terms of statistical information of the previous weight distribution. This involves a third moment which would have to be expressed in terms of the distribution of the finite weight. A more ambitious question is, does the degree distribution with a reasonable scaling converge for large $N$?


	\section*{Acknowledgements}
	
	The second author would like to thank Allen Knutson for many discussions and his hospitality during his stay at the Cornell University.

  \bibliography{biblio}
  \bibliographystyle{amsalpha}	

\end{document}